\documentclass{amsart}

\usepackage{amssymb, amsmath}
\usepackage{mathabx}
\usepackage{graphicx}
\usepackage[utf8]{inputenc}
\usepackage{todonotes}
\usepackage{young}
\usepackage{epstopdf}
\usepackage{mathtools}
\usepackage{diagbox}
\usepackage{xcolor, colortbl}
\usepackage{pifont}

\newtheorem{thm}{Theorem}
\newtheorem{cor}[thm]{Corollary}
\newtheorem{lem}[thm]{Lemma}
\newtheorem{prop}[thm]{Proposition}

\newtheorem{rem}[thm]{Remark}
\newtheorem{ex}{Example}
\numberwithin{equation}{section}

\newcommand{\Z}{\mathbb{Z}}
\newcommand{\N}{\mathbb{N}}
\newcommand{\tq}{\ |\ }
\newcommand{\weyl}{\mathcal{W}}

\newcommand{\xmark}{{\footnotesize\ding{55}}}
\newcommand{\cmark}{\ding{51}}
\newcommand{\cdotsss}{\cdot\!\cdot\!\cdot}

\ysetshade{gray!40}

\begin{document}

\title{Covering relations of k-Grassmannian permutations of type B}

%    author one information
\author{Lonardo Rabelo}
\address{Department of Mathematics, Federal University of Juiz de Fora, Juiz de Fora 36036-900, Minas Gerais, Brazil}
\email{lonardo@ice.ufjf.br}
%\thanks{Supported by FAPESP grant number 08/04628-6}

%    author two information
\author{Jordan Lambert}
\address{Department of Mathematics, Federal University of Juiz de Fora, Juiz de Fora 36036-900, Minas Gerais, Brazil}
\email{jordansilva2008@gmail.com}
%\thanks{Supported by FAPESP grant number 13/10467-3 and 14/27042-8}

% \subjclass is required.


\date{}

\begin{abstract}
The main result of this work is the characterization of the covering relations of the Bruhat order of the maximal parabolic quotients of type B. Our approach is mainly combinatorial and is based in the pattern of the corresponding permutations also called signed $k$-Grassmannians permutations. We obtain that a covering relation can be classified in four different pairs of permutations. This answers a question raised by Ikeda and Matsumura providing a nice combinatorial model for maximal parabolic quotients of type B.
\end{abstract}

\maketitle

\section*{Introduction}

This work focuses in the study of the Bruhat order of the maximal parabolic quotients of type $B$. Some of the main papers providing a combinatorial approach to this subject is that of Dheodhar \cite{Dhe} which gives reduced decompositions while Stanley \cite{Sta} characterizes the pattern of the permutations for elements of these quotients. We may also obtain these patterns in the work of Papi \cite{Papi}.

We are interested in obtaining the covering relations in this context. Notice that a covering relation is the occurrence of pairs which are comparable by the Bruhat order and have length difference equals to one. We know that parabolic quotients inherit the Bruhat order of the Weyl group by projection (for details, see Bjorner-Brenti \cite{Bjo} and also Stembridge \cite{stemb}). Furthermore, the covering relations for all classical Weyl groups may be found in the work of Federico \cite{federico} which in particular includes the type B. However, it is not clear the behaviour of these relations in the quotients. 

To answer this question, we use the well-known characterization of these elements in terms of signed $k$-Grassmannian permutations. Once the Weyl group reflections are explicitly given, we investigate how they act over the permutations case-by-case. We conclude that a covering relation can be sorted in four different classes of pair of permutations. It worth noticing that it is possible because we have a nice formula to compute the length of these permutations.

%The main result of this work is the characterization of the covering relations of the Bruhat order of the maximal parabolic quotients of type B. Our approach is mainly combinatorial and is based in the pattern of the corresponding permutations also called $k$-Grassmannians permutations. Once the Weyl group reflections are known, we investigate how they act over the permutations case-by-case. It worth noticing that it is possible because we have a nice formula to compute the length of these permutations. As a conclusion we obtain that a covering relation can be sorted in four different classes of pair of permutations.
% 
Our motivation is geometric by the fact that such quotients parametrize the Schubert varieties that equip the isotropic Grassmannians (real and complex) with a cellular structure. In the particular context of real Grassmannians, the non-zero coefficients for the boundary map of the cellular (co-)homology occur for these covering pairs. This issue is approached in another paper by the same authors \cite{JL}.

According to Buch-Kresch-Tamvakis \cite{BKT}, there is a bijective correspondence between $k$-Grassmannian permutations and the so called $k$-strict partitions which also give certain kind of Young diagrams. This approach has been very useful for several computations in the algebraic geometric setting (for instance, \cite{BKT}). The work of Ikeda-Matsumura \cite{IM} describes the covering relations according to the weak Bruhat order while it leaves open the characterization of these relations for the Bruhat order. In this former work, there is a hint indicating that the model of Maya diagrams could be useful to approach this problem. So, we also present a version of our result in terms of this class of diagrams which indeed gives some symmetry and give a good picture of the pairs for the covering relations. 

This work is organized as follows: In Section \ref{sec:grass} we introduce the main ingredients about the Grassmannian permutations. In Section \ref{sec:bruhat}, we state and prove our main result in the term of the four types of covering pairs. In Section \ref{sec:maya}, we show how the Maya diagrams are useful to provide a nice picture of the covering relations. 

\section{Grassmannian permutations}\label{sec:grass}

We let $\N=\{1,2,3, \dots\}$ and $\Z$ be the set of integers. For $n,m\in\Z$, where $n\leqslant m$, denote the set $[n,m]=\{n,n+1, \dots, m\}$. For $n\in \N$, denote $[n]=[1,n]$.

The Weyl group $\mathcal{W}_{n}$ of type $B$, also called hyperoctahedral group, is generated by $\Sigma=\{s_i:i=0,1, \ldots, n-1\}$ for which we have the following relations
\begin{align*}
s_i^2&=1 \,,\, i\geq 0; \\
s_0s_1s_0s_1 &= s_1s_0s_1s_0; \\
s_{i+1}s_i s_{i+1} &= s_i s_{i+1}s_i\,,\, 1\leqslant i< n-1;  \\
s_is_j &= s_js_i\,,\, |i-j|\geqslant 2.
\end{align*}

The length $\ell(w)$ of $w \in \mathcal{W}_{n}$ is the minimal number of $s_i$'s in a decomposition of $w$ in terms of the generators. In this case, we say that this is a reduced decomposition of $w$. Each $s_i$, $i\geq 0$, is called a simple reflection. The other reflections are those conjugate to some $s_i$. 

There is a partial order in $\mathcal{W}_{n}$ called the Bruhat-Chevalley order. We say that $w'\leq w$ if given a reduced decomposition $w=s_{i_1}\!\cdots s_{i_{\ell(w)}}$ then $w'=s_{i_{j_{1}}}\!\!\cdots s_{i_{j_{k}}}$ for some indices $1\leq j_{1}<\cdots <j_{k} \leq \ell(w)$ (this is called the ``Subword Property''). It is known that $\mathcal{W}_{n}$ has a maximum element $w_0$ which is an involution, i.e, $w_0^2=1$.

Let $w,w'  \in \mathcal{W}_{n}$ with $\ell(w)=\ell(w')+1$, i.e., if $w=s_{i_{1}}\!\cdots s_{i_{\ell(w)}}$ is reduced decomposition then $w'=s_{i_{1}}\!\cdots \widehat{s_{i_{j}}}\! \cdots s_{i_{\ell(w)}}$ is a reduced decomposition as well. By the Exchange Property of Coxeter groups, there is a reflection $t \in \mathcal{W}_{n}$ (not necessarily simple) such that $w'=wt$ (see \cite{Bjo}, Theorem 2.2.2). In this case, we say that it is a covering relation where $w$ covers $w'$.

Consider the set of all barred permutations $w$ of the form
\begin{align*}
\overline{n},\overline{n-1}, \ldots, \overline{1},0,1, \ldots, n-1, n
\end{align*} 
using the bar to denote a negative sign, and we take the natural order on them, as above. The hyperoctahedral group $\weyl_{n}$ is composed by the barred permutations subject to the relation $\overline{w(i)}=w(\overline{i})$, for all $i$. Then, it is usual to denote $w$ in one-line notation by the sequence $w(1)\cdots \, w(n)$ of positive positions. However, we also could use the full description of $w$ with the corresponding negative reflections, if it require so.

The simple reflections are 
\begin{align*}
s_0&=(\overline{1},1);\\
s_i&=(\overline{i+1},\overline{i})(i,i+1)\, , \, \mbox{ for all } i\geq 1.
\end{align*}

If we think the elements of $\mathcal{W}_{n}$ acting at right over the permutations, we have that $s_0$ changes the sign in the first position and $s_i$ changes the entries in the positions $i$ and $i+1$, i.e., in the one-line notation
\begin{align}
w(1)w(2)\cdots w(n) \cdot s_0 &= \overline{w(1)} w(2)\cdots w(n); \\
w(1)\cdots w(i) w(i+1) \cdots w(n) \cdot s_i &= w(1)\cdots w(i+1) w(i) \cdots w(n).
\end{align}

By Proposition 8.1.5 of \cite{Bjo}, the set of reflections is determined by
\begin{align}\label{eq:reflections}
\{(i,j)(\overline{i},\overline{j}) \mid 1\leq i< |j| \leq n\} \bigcup \{ (i,\overline{i}) \mid 1\leq i \leq n\}.
\end{align}

Hence the action of a reflection $(i,j)(\overline{i},\overline{j})$, for some $1\leq i< |j| \leq n$, will permute the entries in positions $i$ and $j$. The reflection $(i,\overline{i})$, for some $1\leq i \leq n$, changes the sign at $i$-th position.

The length of $w \in \mathcal{W}_{n}$ is given by the following formula (see \cite{Bjo}, Eq. (8.3))
\begin{align}\label{eq:length}
\ell(w) = \mathrm{inv}(w(1),\ldots, w(n)) - \sum_{ \mathclap{\{j\mid w(j)<0\}}}\ w(j)
\end{align}
where 
\begin{align*}
\mathrm{inv}(w(1),\ldots,w(n)) = \# \{(i,j)\mid 1\leq i<j\leq n, w(i)>w(j)\}. 
\end{align*}

The following result will be very useful in the sequel about the covering relations in $\mathcal{W}_{n}$ and the Exchange Property stated in terms of the reflections given by (\ref{eq:reflections}).

\begin{prop}[\cite{Bjo}, Prop. 8.1.6]\label{prop:bruhatBjor}
Let $w,w' \in \weyl_{n}$. Then, the following are equivalent
\begin{enumerate}
\item[(i)] $w'\leqslant w$ with $\ell(w)=\ell(w')+1$.
\item[(ii)] There exist $\overline{n} \leqslant i < j \leqslant n$, such that $w(i)>w(j)$ and
\begin{itemize}
\item $w'=w\cdot (i,j)(-i,-j)$, if $|i|\neq |j|$, or
\item $w'=w\cdot (i,j)$, if $|i|=|j|$.
\end{itemize}
\end{enumerate}
\end{prop}

For each $0\leq k\leq n$, define the set $(k)=\Sigma-\{s_{k}\}$ of simple reflections without $s_{k}$. The corresponding (parabolic) subgroup $\mathcal{W}_{(k)}$ is generated by $s_i$, with $i\neq k$. Notice that $\mathcal{W}_{(k)}\cong \mathcal{W}_k \times S_{n-k}$, where $\mathcal{W}_k$ is the subgroup generated by $s_i$, $0\leq i \leq k$. 
For any $w=w(1)\cdots w(n)$, it follows that its coset $w\mathcal{W}_{(k)}$ is composed by permutations with its first $k$ entries permuted with signs changed -- corresponding to the $\mathcal{W}_k$ part -- together with the permutations with the remaining $(n-k)$ permuted -- but without change of signs corresponding to the $S_{n-k}$ part. 
 
The set of the minimum-length coset representatives for $\mathcal{W}_{n}/\mathcal{W}_{(k)}$ is defined by
\begin{align*}
\mathcal{W}^{(k)}_{n} = \{ w \in \mathcal{W}_{n} \mid \ell(w)<\ell(ws_i), \forall i\geq 0, i\neq k\}.
\end{align*}
Indeed, there exists a unique minimal length element in each coset $w\weyl_{(k)}$.

Now, we can give an explicit description of these representatives according to the above description of the cosets $w\mathcal{W}_{(k)}$. In fact, by Equation (\ref{eq:length}), we must seek inside each coset the elements with minimal number of inversions and negative numbers. In the first $k$ entries, we always may take only positives elements ordered since it gives the least contribution to the length while for the remaining $n-k$ entries we only order them since we cannot avoid negative entries. So, we can conclude that the one-line notation of $w$ can be identified by the form
\begin{align}\label{eq:kgrass}
w = w_{u,\lambda} = u_1\, \cdots u_k | \overline{\lambda_r}\, \cdots\, \overline{\lambda_1}\, v_1\, \cdots \, v_{n-k-r}
\end{align}
where 
\begin{align}\label{eq:kgrass_conds}
0<u_1<\cdots <u_k\ , &&& u_i=w(i), \mbox{ for } 1\leq i  \leq k; \nonumber\\
0<\lambda_1 < \cdots < \lambda_r\ ,  &&& \overline{\lambda_i} = w(k+r-i+1), \mbox{ for } 1\leq i \leq r; \\
0<v_1 <\cdots < v_{n-k-r}\ ,  &&& v_i = w(k+r+i), \mbox{ for } 1\leq i \leq n-k-r. \nonumber
\end{align}

They are called $k$-Grassmannian permutations.
\begin{rem} Proposition \ref{prop:bruhatBjor} still holds for elements of $\mathcal{W}^{(k)}_{n}$ since $\mathcal{W}^{(k)}_{n} \subset \mathcal{W}_{n}$ and the projection $\pi: \mathcal{W}_{n}\rightarrow \mathcal{W}^{(k)}_{n}$ preserves the Bruhat order (see \cite{stemb}, Prop. 1.1).
\end{rem}
We now define a pair of double partitions $\alpha$  and $\lambda$ associated with each $k$-Grassmannian permutation $w$ given by Equation (\ref{eq:kgrass}). The negative part of $w$ provides us a strict partition $\lambda =(\lambda_{r}>\cdots >\lambda_{1}>0)$. For each $i$, $0< i\leqslant k$, we define
\begin{align}\label{eq:alpha1}
\alpha_i = u_i - i + d_i \, , \, \mbox{ with } d_i = \# \{\lambda_{j} \mid \lambda_j > u_i\}.
\end{align}

We claim that $\alpha=(n-k\geq \alpha_k\geq\alpha_{k-1}\geq \cdots \geq \alpha_1\geq0)$ is a partition.
Indeed, for each $i$, we may collect all indexes that are greater than $u_i$ that appear in the permutation (\ref{eq:kgrass}) accordingly to the position they occupy by the following formula
\begin{align*}
[n] - [u_{i}]=\{u_{j} \mid u_j>u_i \} \cup \{\lambda_{j} \mid \lambda_j >u_i\} \cup \{v_{j} \mid v_j >u_i\}
\end{align*}
where the cardinality is given by
\begin{align}\label{eq:lessthanui}
u_i = n- \#\{u_{j} \mid u_j>u_i \} - \#\{\lambda_{j} \mid \lambda_j >u_i\} - \#\{v_{j} \mid v_j >u_i\}.
\end{align}
Now, we observe that $\#\{u_{j} \mid u_j>u_i \}=k-i$. It follows that Equation (\ref{eq:lessthanui}) is equivalent to
\begin{align}\label{eq:ui}
u_i = n-k+i- d_{i} -\mu_{i}.
\end{align}
where $\mu_{i}=\#\{v_{j}\mid v_j>u_i\}$. Then, Equation (\ref{eq:alpha1}) may be rewritten as 
\begin{align}\label{eq:alpha3}
\alpha_i = n-k - \mu_{i}.
\end{align}
It is clear now from (\ref{eq:alpha3}) that $\alpha$ is a partition inside a $k\times (n-k)$ rectangle.  Denote $|\alpha|=\sum_{i=1}^{k} \alpha_i$ and $|\lambda|=\sum_{i=1}^{r} \lambda_i$. Observe that $\mu_{i}=\mu_{i}(w)$ also depends on the choice of $w$.

\begin{lem}\label{lem:dim} Let $w \in \mathcal{W}_{n}^{(k)}$. The length $\ell(w)$ of $w$ is given by the sum of entries the pair of double partition $\alpha, \lambda$, i.e., 
\begin{align*}
\ell(w) = |\alpha|+|\lambda|.
\end{align*}
\end{lem}
\begin{proof}
The sum of $\lambda$'s corresponds to the sum $- \sum_{ \{j\mid w(j)<0\} } w(j)$ in Equation (\ref{eq:length}). It remains to show that the inversions of $(w(1),\ldots, w(n))$ are given by the sum of the $\alpha$'s. Since there is a unique descent in position $k$, all inversions corresponds to inversions among the $u_i$'s with all $\lambda$'s and some of $v_j$'s. Then, for each $1\leqslant i \leqslant k$, we have that the number of inversion related to $u_{i}$ is $n-k-\#\{v_{j}\tq u_{i}< v_{j}\}=\alpha_{i}$, by Equation (\ref{eq:alpha3}). Hence, $\mathrm{inv}(w(1),\ldots,w(n)) = |\alpha|$.
\end{proof}

\section{Bruhat order of Grassmannian permutations}\label{sec:bruhat}

Let $w$ and $w'$ be permutations in $\weyl^{(k)}_{n}$ written in one-line notation as
\begin{align*}
w &= u_1\, \cdots\, u_k| \overline{\lambda_r}\, \cdots\, \overline{\lambda_1}\, v_1 \,\cdots\, v_{n-k-r}, \\
w' &= u_1'\, \cdots\, u_k'| \overline{\lambda_{r'}'}\, \cdots\, \overline{\lambda_1'}\, v_1' \,\cdots\, v_{n-k-r'}'.
\end{align*}
Denote by $\alpha,\lambda$ and $\alpha',\lambda'$ the pair of partitions associated with $w$ and $w'$, respectively. Also denote $\mu_{i}=\mu_{i}(w)$ and $\mu'_{i}=\mu_{i}(w')$.

We call:

\begin{itemize}
\item $w,w'$ a pair of \emph{type B1} if
\begin{align*}
w &= \cdots\, |\, \cdots\,  \overline{1}\,  \cdots \, & \mbox{ and} &&
w' &= \cdots\, |\, \cdots\,  1\,  \cdots.
\end{align*}
In other words, we choose $w$ such that $\lambda_{1}=1$, and $w'$ is obtained from $w$ by removing the negative sign from $\overline{1}$.

\vspace*{0.2cm}
\item $w,w'$ a pair of \emph{type B2} if
\begin{align*}
w &= \cdots\, |\, \cdots\,  \overline{a}\,  \cdots \, (a-1)\, \cdots \,& \mbox{ and} &&
w' &= \cdots\, |\, \cdots\, \overline{a-1}\,  \cdots \, a\, \cdots,
\end{align*}
where $a>0$. In other words, there are $t\in [r]$ and $q \in [n-k-r]$ such that $\lambda_{t}=a$ and $v_{q}=a-1$, and $w'$ is obtained from $w$ by switching $v_{q}$ and $\lambda_{t}$.

\vspace*{0.2cm}
\item $w,w'$ a pair of \emph{type B3} if
\begin{align*}
w &= \cdots\, a \,\cdots\, |\,  \cdots \, (a-x)\, \cdots \, & \mbox{ and} &&
w' &= \cdots\,  (a-x) \,\cdots \,|\, \cdots\, a\, \cdots,
\end{align*}
where $a > x > 0$. In other words, there are $p\in [k]$ and  $q \in [n-k-r]$ such that $u_{p}=a$ and $v_{q}=a-x$. The permutation $w'$ is obtained from $w$ by switching $u_{p}$ and $v_{q}$.
%where $a > x \geqslant 1$. In other words, there are $p\in [k]$, $q \in [n-k-r]$, and $t\in [r]$ such that $u_{p}=a$, $v_{q}=a-x$, and $\lambda_{t}=a-x+1, \lambda_{t+1}=a-x+2, \dots, \lambda_{t+x-2}=a-1$. The permutation $w'$ is obtained from $w$ by switching $u_{p}$ and $v_{q}$.

\vspace*{0.2cm}
\item $w,w'$ a pair of \emph{type B4} if
\begin{align*}
w &= \cdots\, (a-x) \,\cdots\, |\,  \cdots \, \overline{a}\, \cdots \, & \mbox{ and} &&
w' &= \cdots\,  a \,\cdots \,|\, \cdots\, \overline{a-x}\, \cdots,
\end{align*}
where $a > x  > 0$. In other words, there are $p\in [k]$ and $t \in [r]$ such that $u_{p}=a-x$ and $\lambda_{t}=a$. The permutation $w'$ is obtained from $w$ by switching $u_{p}$ and $\lambda_{t}$.

\end{itemize}

The following lemma states a property for pairs of type B3 and B4.

\begin{lem}\label{lem:B3B4}
Let $w,w'\in \weyl_{n}^{(k)}$.
\begin{enumerate}
\item [i] If $w,w'$ is a pair of type B3 then all positive values $a-x+1, a-x+2, \dots, a-1$ should belong to the $\lambda$'s, i.e., there is $t\in[r]$ such that $\lambda_{t}=a-x+1, \lambda_{t+1}=a-x+2, \dots, \lambda_{t+x-2}=a-1$;
\item [ii] If $w,w'$ is a pair of type B4 then all positive values $a-x+1, a-x+2, \dots, a-1$ should belong to the $v$'s, i.e., there is $q \in [n-k-r]$ such that $v_{q}=a-x+1, v_{q+1}=a-x+2, \dots, v_{q+x-2}=a-1$.
\end{enumerate}
\end{lem}
\begin{proof} Let $w,w'$ be a pair of type $B3$. Let $p\in [k]$ be such that $u_{p}=a$. Suppose there is $i\in [k]$ such that $i<p$ and $u_{i}\in [a-x+1,a-1]$. After swapping $a$ and $a-x$, we have that $u_{p}'=a-x$ while $u_{i}'$ still belongs to $[a-x+1,a-1]$. In other words, $u_{i}'>u_{p}'$ for $i<p$, which is impossible by Equations \eqref{eq:kgrass_conds}. Now, suppose there is $i\in [k]$ such that $p<i$ and $u_{i}\in [a-x+1,a-1]$. It gives that $u_p>u_i$ for $p<i$, which is not possible by Equations (\ref{eq:kgrass_conds}). Hence, there is no $i\in [k]$ such that $u_{i}\in [a-x+1,a-1]$. Likewise, we prove that there is no $j\in [n-k-r]$ such that $v_{j}\in [a-x+1,a-1]$. Hence, we have statement (1).

We use the same idea to prove statement (2).
\end{proof}

To define the previous four types of pairs, we only require that both $w$ and $w'$ belong to $\weyl_{n}^{(k)}$. In principle, it is not clear the relationship between them. This is the content of our main theorem.

\begin{thm}\label{thm:main}
Let $w,w'\in \weyl_{n}^{(k)}$. Then $w$ covers $ w'$ if, and only if, $w,w'$ is a pair of type B1, B2, B3 or B4.
\end{thm}
\begin{proof}
Consider the sets of positive integers $I_{1}=[k]$, $I_{2}=[k+1,r]$ and $I_{3}=[k+r+1,n]$, and their respective sets of negative integers $I_{\overline{1}}=[\overline{k},\overline{1}]$, $I_{\overline{2}}=[\overline{r},\overline{k+1}]$ and $I_{\overline{3}}=[\overline{n},\overline{k+r+1}]$. Notice that each set represent a block of positions in $w$ as shown below
\begin{align*}
w =
\underbracket{\overline{v_{n-k-r}}\, \cdots \, \overline{v_{1}}}_{I_{\overline{3}}}\,
\underbracket{\lambda_{1}\, \cdots \, \lambda_{r}}_{I_{\overline{2}}}|
\underbracket{\overline{u_{k}}\, \cdots\, \overline{u_{1}}}_{I_{\overline{1}}}\, 0\, 
\underbracket{u_1\, \cdots\, u_k}_{I_{1}}| 
\underbracket{\overline{\lambda_r}\, \cdots\, \overline{\lambda_1}}_{I_{2}}\,
\underbracket{v_1 \,\cdots\, v_{n-k-r}}_{I_{3}}.
\end{align*}

Explicitly,
\begin{align}
w(i) = 
\left\{
\begin{array}{cc}
v_{i-k-r} & \text{, for } i\in I_{3}; \\ 
\overline{\lambda_{k+r+1-i}} & \text{, for } i\in I_{2}; \\ 
u_{i} & \text{, for } i\in I_{1}; \\
0 & \text{, for } i=0;\\
\overline{u_{-i}} & \text{, for } i\in I_{\overline{1}}; \\ 
\lambda_{k+r+1+i} & \text{, for } i\in I_{\overline{2}}; \\ 
\overline{v_{-(i+k+r)}} & \text{, for } i\in I_{\overline{3}}.
\end{array} 
\right.
\end{align}

Suppose that $w$ covers $w'$, i.e., $w'\leqslant w$ and $\ell(w)=\ell(w')+1$. Proposition \ref{prop:bruhatBjor} says that there are $\overline{n} \leqslant i<j \leqslant n$ such that $w(i)>w(j)$ and either $w'=w\cdot (i,j)(-i,-j)$, if $|i|\neq |j|$, or $w'=w\cdot (i,j)$, if $|i|=|j|$.

First of all, $i=0$ if, and only if, $j=0$ since the symmetry of $\weyl_{n}^{(k)}$ implies that $w(0)=0$. Then, we will always consider non-zero $i$ and $j$.

The set $[-n,n]$ is the disjoint union $I_{\overline{3}}\cup I_{\overline{2}}\cup I_{\overline{1}} \cup \{0\}\cup I_{1} \cup I_{2} \cup I_{3}$. We will prove the theorem by checking all possible combinations of $i<j$ such that $i\in I_{m}$ and $j\in I_{l}$, for $m,l\in\{\overline{3},\overline{2},\overline{1},1,2,3\}$. Table \ref{tbl:ijchoices} encloses such information for $i$ and $j$, where the rows denote $I_{m}$ and the column denote $I_{l}$.

\begin{table}[ht]
\def\arraystretch{1.3}
\caption{Possible choices for $i<j$ such that $w(i)>w(j)$.}\label{tbl:ijchoices}
\begin{tabular}{|c|c|c|c|c|c|c|}
\hline
\diagbox{$i$}{$j$} & $I_{3}$ & $I_{2}$ & $I_{1}$ & $I_{\overline{1}}$ & $I_{\overline{2}}$ & $I_{\overline{3}}$ \\ 
\hline
$I_{3}$ & \xmark &  &  &  &  &  \\ 
\hline 
$I_{2}$ & \xmark & \xmark &  &  &  &  \\ 
\hline 
$I_{1}$ & \cellcolor{lightgray}\cmark  & \cellcolor{lightgray}\cmark & \xmark &  &  &  \\ 
\hline 
$I_{\overline{1}}$ & \xmark & \cellcolor{lightgray}\cmark  & \xmark & \xmark &  &  \\ 
\hline 
$I_{\overline{2}}$ & \cellcolor{lightgray}\cmark  & \cellcolor{lightgray}\cmark  & \cmark  & \cmark  & \xmark &  \\ 
\hline 
$I_{\overline{3}}$ & \xmark & \cmark  & \xmark & \cmark  & \xmark & \xmark \\ 
\hline 
\end{tabular} 
\end{table}

There are some choices of $i<j$ for which the relation $w(i)>w(j)$ is not satisfied. For instance, for every $i,j\in I_{3}$ such that $i<j$, we have $w(i)=v_{i-k-r} < v_{j-k-r}=w(j)$.
Empty cells in Table \ref{tbl:ijchoices} means $i>j$. A cell marked with \xmark\ means that for all $i<j$ in the respective set, we have $w(i)<w(j)$. Cells marked with \cmark\ are the ones such that we could have $i<j$ satisfying $w(i)>w(j)$. 

When $|i|\neq |j|$, we know $w'$ is obtained by swapping the values $w(i)$ and $w(j)$, and also swapping the values $w(-i)$ and $w(-j)$. Then, there is a symmetry in the choice of $i,j$. For instance, choosing $i<j$ such that $i\in I_{\overline{2}}$ and $j\in I_{\overline{1}}$ is equivalent to choose $-j<-i$ such that $-j \in I_{1}$ and $-i \in I_{2}$. It is enough to verify the case where $i$ and $j$ belongs to $I_{1}$ and $I_{2}$, respectively.

Therefore, we only have to check the five possibilities in Table \ref{tbl:ijchoices} represented by the shaded cells marked with \cmark. 

\textbf{1st case:}
Suppose that $i\in I_{1}$ and $j\in I_{3}$ such that $w(i)>w(j)$. Let $a>x>0$ be integers such that $w(i)=u_{i}=a$ and $w(j)=v_{j-k-r}=a-x$. The permutation $w'$ is obtained from $w$ by swapping $w(i)=a$ and $w(j)=a-x$, and swapping the respective negatives $w(-i)=\overline{a}$ and $w(-j)=\overline{a-x}$.
In short
\begin{align*}
w & = u_1 \, \cdots\, a \, \cdots\, u_k| \overline{\lambda_r}\, \cdots\, \overline{\lambda_1}\, v_1 \, \cdots\, (a-x) \, \cdots\, v_{n-k-r}, \\
w' &= u_1\, \cdots\, (a-x) \, \cdots\, u_k| \overline{\lambda_r}\, \cdots\, \overline{\lambda_1}\, v_1 \,\cdots\, a\, \cdots\, v_{n-k-r}.
\end{align*}

Let us compare the lengths of $w$ and $w'$. Lemma \ref{lem:B3B4}(i) says that all integers $a-x+1, \dots, a-1$ belongs to the $\lambda$'s. Then
\begin{align*}
\{v_{q}'\mid v_q'>u_i'\} & = \{v_{q}' \mid v_q'>a-x\} = \{v_{j-k-r}', v_{j-k-r+1}',\dots, v_{n-k-r}'\} \\
 & = \{a\} \cup \{v_{j-k-r+1},\dots, v_{n-k-r}\} = \{a\} \cup \{v_{q}\mid v_q>u_i\}
\end{align*}
and
\begin{align*}
\{v_{q}\mid v_q'>u_t'\} = \{v_{q}\mid v_q>u_i\}\mbox{ , for } t \in [k], t\neq i.
\end{align*}

Hence $\mu'_{i}=\mu_{i}+1$, and  $\mu'_{t}=\mu_{t}$ for $t\neq i$. It follows from Equation \eqref{eq:alpha3} that $\alpha'_{i}=\alpha_{i}-1$ and $\alpha'_{t}=\alpha_{t}$ for $t\neq i$, which implies that $|\alpha|=|\alpha'|+1$. Since $|\lambda|=|\lambda'|$, it is clear that $\ell(w)=\ell(w')+1$.

Therefore, $w,w'$ is a pair of type B3.

\vspace*{0.2cm}

\textbf{2nd case:}
Suppose that $i\in I_{1}$ and $j\in I_{2}$ such that $w(i)>w(j)$. Observe that if we swap $w(i)=u_{i}$ and $w(j)=\overline{\lambda_{j-k-r}}$, it would put a negative entry in the first $k$ positions of $w'$, which is not allowed by Equations (\ref{eq:kgrass_conds}). Therefore, this is not a valid case.

\vspace*{0.2cm}

\textbf{3rd case:}
Suppose that $i\in I_{\overline{1}}$ and $j\in I_{2}$ such that $w(i)>w(j)$. Let $a>x>0$ be integers such that $w(j)=\overline{\lambda_{k+r+1-j}}=\overline{a}$ and $w(i)=\overline{u_{-i}}=\overline{a-x}$. The permutation $w'$ is obtained from $w$ by swapping $w(i)=\overline{a-x}$ and $w(j)=\overline{a}$ and the respective negatives $w(-i)=a-x$ and $w(-j)=a$. In short
\begin{align*}
w & = u_1 \, \cdots\, (a-x) \, \cdots\, u_k| \overline{\lambda_r}\, \, \cdots\, \overline{a} \cdots \, \overline{\lambda_1}\, v_1 \, \cdots\, v_{n-k-r}, \\
w' & = u_1 \, \cdots\, a \, \cdots\, u_k| \overline{\lambda_r}\, \, \cdots\, \overline{a-x} \cdots \, \overline{\lambda_1}\, v_1 \, \cdots\, v_{n-k-r}.
\end{align*}

Let us compare the lengths of $w$ and $w'$. Lemma \ref{lem:B3B4} says that all integers $a-x+1, \dots, a-1$ are in the $v$'s. Then
\begin{align*}
\{v_{q}\mid v_q>u_{-i}\} & = \{v_{q}\mid v_q>a-x\} = \{a-x+1, \dots, a-1\} \cup\{v_{q}\mid v_q>a\} \\
 & = \{a-x+1, \dots, a-1\} \cup\{v_{q}'\mid v_q'>u_{-i}'\}
\end{align*}
and
\begin{align*}
\{v_{q}\mid v_q'>u_t'\} = \{v_{q}\mid v_q>u_i\} \quad \mbox{ , for } t\in [k], t\neq -i.
\end{align*}

Hence $\mu_{-i}=\mu_{-i}'+(x-1)$, and  $\mu'_{t}=\mu_{t}$ for $t\neq -i$. Then, it follows from Equation \eqref{eq:alpha3} that $\alpha'_{-i}=\alpha_{-i}+(x-1)$ and $\alpha'_{t}=\alpha_{t}$ for $t\neq -i$, which implies that $|\alpha|=|\alpha'|-x+1$. Since $|\lambda|=|\lambda'|+x$, it is clear that $\ell(w)=\ell(w')+1$.

Therefore, $w,w'$ is a pair of type B4.

\vspace*{0.2cm}

\textbf{4th case:}
Suppose that $i\in I_{\overline{2}}$ and $j\in I_{3}$ such that $w(i)>w(j)$. Let $a>x>0$ be integers such that $w(i)=\lambda_{k+r+1+i}=a$ and $w(j)=v_{j-k-r}=a-x$. The permutation $w'$ is obtained from $w$ by swapping $w(i)=a$ and $w(j)=a-x$, and swapping the respective negatives $w(-i)=\overline{a}$ and $w(-j)=\overline{a-x}$. In short
\begin{align*}
w & = u_1\, \cdots\, u_k| \overline{\lambda_r}\, \cdots\, \overline{a} \, \cdots\, \overline{\lambda_1}\, v_1 \, \cdots\, (a-x) \, \cdots\, v_{n-k-r}, \\
w' &= u_1\, \cdots\, u_k| \overline{\lambda_r}\, \cdots\, \overline{a-x} \, \cdots\, \overline{\lambda_1}\, v_1 \,\cdots\, a\, \cdots\, v_{n-k-r}.
\end{align*}

Let us compare the lengths of $w$ and $w'$. Since $w,w'\in \weyl_{n}^{(k)}$, all integers $a-x+1, \dots, a-1$ should be in the $u$'s (this can be proved likewise in Lemma \ref{lem:B3B4}). Denote by $p\in [0,k]$ the largest integer such that $u_{p}<a-x$ (if required, take $u_{0}=0$). Clearly, $u_{p+1}=a-x+1, \dots, u_{p+x-1}=a-1$. Then
\begin{align*}
\{v_{q}'\mid v_q'>u_t'\} & = \{v_{q}\mid v_q>u_i\} \quad \mbox{ , for every } t\in [p];\\
\{v_{q}'\mid v_q'>u_t'\} & = \{v_{q}\mid v_q>u_i\}\cup\{a\} \quad \mbox{ , for every } t\in [p+1, p+x-1];\\
\{v_{q}'\mid v_q'>u_t'\} & = \{v_{q}\mid v_q>u_i\} \quad \mbox{ , for every } t\in [p+x,k].
\end{align*}
and, hence,
\begin{align*}
\mu_{t}' = \mu_{t} +
\left\{
\begin{array}{cl}
1 & \mbox{ , if } p< t \leqslant p+x-1;\\
0 & \mbox{ , otherwise.}
\end{array} 
\right.
\end{align*}

By Equation \eqref{eq:alpha3}, 
\begin{align*}
\alpha_{t} = \alpha_{t}' +
\left\{
\begin{array}{cl}
1 & \mbox{ , if } p< t \leqslant p+x-1;\\
0 & \mbox{ , otherwise.}
\end{array} 
\right.
\end{align*}
and $|\alpha|=|\alpha'|+(x-1)$. Clearly, $|\lambda|=|\lambda'|+x$, which implies that $\ell(w)=\ell(w')+2x-1$. By hypothesis, $2x-1$ should be equal to $1$, which lead us to conclude that $x=1$.

Therefore, $w,w'$ is a pair of type B2.

\vspace*{0.2cm}

\textbf{5th case:}
Suppose that $i\in I_{\overline{2}}$ and $j\in I_{2}$ such that $w(i)>w(j)$. First of all, assume $-i < j$. Let $a>x > 0$ be integers such that $w(i)=\lambda_{k+r+1+i}=a$ and $w(j)=\overline{\lambda_{k+r-j}}=\overline{a-x}$. The permutation $w'$ is obtained from $w$ by swapping $w(i)=a$ and $w(j)=\overline{a-x}$, and the respective negatives $w(-i)=\overline{a}$ and $w(-j)=a-x$. In short
\begin{align*}
w & = u_1\, \cdots\, u_k| \overline{\lambda_r}\, \cdots\, \overline{a} \, \cdots\, \overline{a-x}\, \cdots\, \overline{\lambda_1}\, v_1 \,\cdots\, v_{n-k-r}, \\
w' &= u_1\, \cdots\, u_k| \overline{\lambda_r}\, \cdots\, (a-x) \, \cdots\, a\, \cdots\, \overline{\lambda_1}\, v_1 \,\cdots\, v_{n-k-r}.
\end{align*}

Notice that if there is some $\overline{\lambda_{m}}$ to the right of $w'(-i)=a-x$ or $a>v_{1}$ then it would not satisfy Equations \eqref{eq:kgrass_conds}. Hence we have $-i=k+r-1$, $j=k+r$ and $a<v_{1}$, and both $a-x$ and $a$ of $w'$ should be added to the $v$'s. Moreover, all integers $a-x+1, \dots, a-1$ should be in the $u$'s. Denote by $p\in [0,k]$ the largest integer such that $u_{p}<a-x$. Clearly, $u_{p+1}=a-x+1, \dots, u_{p+x-1}=a-1$.

Let us compare the lengths of $w$ and $w'$. We have that
\begin{align*}
\{v_{q}'\mid v_q'>u_t'\} & = \{v_{q}\mid v_q>u_t\}\cup \{a-x,a\} \mbox{ , for every } t\in [p]; \\
\{v_{q}'\mid v_q'>u_t'\} & = \{v_{q}\mid v_q>u_t\}\cup \{a\} \mbox{ , for every } t\in [p+1, p+x-1]; \\
\{v_{q}'\mid v_q'>u_t'\} & = \{v_{q}\mid v_q>u_t\} \mbox{ , for every } t\in [p+x, k],
\end{align*}
and, hence,
\begin{align*}
\mu_{t}' = \mu_{t} +
\left\{
\begin{array}{cl}
2 & \mbox{ , if } t \leqslant p; \\ 
1 & \mbox{ , if } p< t \leqslant p+x-1;\\
0 & \mbox{ , if } p+x-1 < t \leqslant k.
\end{array} 
\right.
\end{align*}

By Equation \eqref{eq:alpha3}, 
\begin{align*}
\alpha_{t} = \alpha_{t}' +
\left\{
\begin{array}{cl}
2 & \mbox{ , if } t \leqslant p; \\ 
1 & \mbox{ , if } p< t \leqslant p+x-1;\\
0 & \mbox{ , if } p+x-1 < t \leqslant k.
\end{array} 
\right.
\end{align*}
and $|\alpha|=|\alpha'|+2p+(x-1)$. Clearly, $|\lambda|=|\lambda'|+(a-x)+a$, which implies that $\ell(w)=\ell(w')+2a+2p-1$. By hypothesis, $2a+2p-1$ should be equal to $1$. This implies that $a=1$, $p=0$, and $x$ is an integer such that $1>x>0$. Therefore, this is not a valid case.

If $-i>j$ then we can proceed as above to show that this is also not a valid case.

Finally, suppose that $-i=j$. Let $a > 0$ be an integer such that $w(i)=\lambda_{k+r+1+i}=a$ and $w(j)=\overline{\lambda_{k+r-j}}=\overline{a}$. The permutation $w'$ is obtained from $w$ by swapping $w(i)=a$ and $w(j)=\overline{a}$. In short
\begin{align*}
w & = u_1\, \cdots\, u_k| \overline{\lambda_r}\, \cdots\, \overline{a} \, \cdots\, \overline{\lambda_1}\, v_1 \,\cdots\, v_{n-k-r}, \\
w' &= u_1\, \cdots\, u_k| \overline{\lambda_r}\, \cdots\, a \, \cdots\, \overline{\lambda_1}\, v_1 \,\cdots\, v_{n-k-r}.
\end{align*}

By Equations \eqref{eq:kgrass_conds}, $w'$ lies in $\weyl_{n}^{(k)}$ if, and only if, $-i=j=k+r$, and $a<v_{1}$. Then, $a$ should be added to the $v$'s. Denote by $p\in [0,k]$ the largest integer such that $u_{p}<a$. Let us compare the length of $w$ and $w'$. We have
\begin{align*}
\{v_{q}'\mid v_q'>u_t'\} & = \{v_{q}\mid v_q>u_t\}\cup\{a\}  \mbox{ , for every } t\in [p]; \\
\{v_{q}'\mid v_q'>u_t'\} & = \{v_{q}\mid v_q>u_t\}  \mbox{ , for every } t\in [p+1,k].
\end{align*}
and, hence,
\begin{align*}
\mu_{t}' = \mu_{t} +
\left\{
\begin{array}{cl}
1 & \mbox{ , if } t \leqslant p; \\ 
0 & \mbox{ , if } p< t \leqslant k.
\end{array} 
\right.
\end{align*}

By Equation \eqref{eq:alpha3}, 
\begin{align*}
\alpha_{t} = \alpha_{t}' +
\left\{
\begin{array}{cl}
1 & \mbox{ , if } t \leqslant p; \\ 
0 & \mbox{ , if } p< t \leqslant k.
\end{array}
\right.
\end{align*}
and $|\alpha|=|\alpha'|+ p$. Clearly, $|\lambda|=|\lambda'|+a$, which implies that $\ell(w)=\ell(w')+a+p$. By hypothesis, $a+p$ should be equal to $1$, which lead us to conclude that $a=1$ and $p=0$.

Therefore, $w,w'$ is a pair of type B1.

\vspace*{0.2cm}
Clearly, we implicitly proved the reciprocal.
\end{proof}

As consequence of Theorem \ref{thm:main}, if one starts with any permutation $w \in \weyl_{n}^{(k)}$, we establish certain conditions to determine all possible $w' \in \weyl_{n}^{(k)}$ covered by $w$.

\begin{cor}[Length-decreasing]\label{coro:dlength} Let $w\in \weyl_{n}^{(k)}$. 
\begin{enumerate}
\item If $w = \cdots\, |\, \cdots\,  \overline{1}\,  \cdots$ then $w$ covers $w' = \cdots\, |\, \cdots\,  1\,  \cdots$;
\item If $w = \cdots\, |\, \cdots\,  \overline{a}\,  \cdots \, (a-1)\, \cdots$ then $w$ covers $w' = \cdots\, |\, \cdots\, \overline{a-1}\,  \cdots \, a\, \cdots$;
\item If $w = \cdots\, a \,\cdots\, |\,  \cdots \, b\, \cdots$ where $a>b$ and all positive integers $b+1, b+2, \dots, a-1$ belong to the $\lambda$'s then $w$ covers $w' = \cdots\,  b \,\cdots \,|\, \cdots\, a\, \cdots$;
\item If $w = \cdots\, b \,\cdots\, |\,  \cdots \, \overline{a}\, \cdots$ where $a>b$ and all positive integers $b+1, b+2, \dots, a-1$ belong to the $v$'s then $w$ covers $w' = \cdots\,  a \,\cdots \,|\, \cdots\, \overline{b}\, \cdots$.
\end{enumerate}
\end{cor}
\begin{proof}
For statements (1) and (2), clearly $w'$ belongs to $\weyl_{n}^{(k)}$, which implies that $w,w'$ are pairs of type B1 or B2, respectively.

For (3), the condition of $b+1, b+2, \dots, a-1$ belong to the $\lambda$'s guarantees that $w'\in\weyl_{n}^{(k)}$. Then, $w,w'$ is a pair of type B3. The same argument holds for (4), concluding that $w,w'$ is a pair of type B4.
\end{proof}

In some sense, Corollary \ref{coro:dlength} combines the results of Theorem \ref{thm:main} and Lemma \ref{lem:B3B4}.

\begin{ex}\label{ex:covering}
Consider $w=2\,5\,6|\overline{8}\,\overline{7}\,\overline{4}\,\overline{1}\,3$ where $n=8$ and $k=3$. 
Let us determine $w'\in \weyl_8^{(3)}$ such that $w$ covers $w'$ following the Corollary \ref{coro:dlength}. It is immediate that $w$ covers $2\,5\,6|\overline{8}\,\overline{7}\,\overline{4}\,1\,3$ which is a pair of type B1.

To get a pair of type B2, we should pick an entry $\overline{a}$ in the $\lambda$'s and an entry $(a-1)$ in the $u$'s. The only possible choice is the pair of entries $\overline{4}$ and $3$ which gives that $w$ covers $2\,5\,6|\overline{8}\,\overline{7}\,\overline{3}\,\overline{1}\,4$.

To get a pair of type B3, we should pick an entry $a$ in the $u$'s and an entry $b$ in the $v$'s such that $a>b$ and all positive integers $b+1, b+2, \dots, a-1$ belong to the $\lambda$'s. Choosing $5$ and $3$ gives a covering since $4$, which is the only integer between $b=3$ and $a=5$, is contained in the $\lambda$'s. Hence $w$ covers $2\,3\,6|\overline{8}\,\overline{7}\,\overline{4}\,\overline{1}\,5$. Choosing $6$ and $3$ does not give a covering since $5$, which is a number between $b=3$ and $a=6$, does not belong to the $\lambda$'s.

Finally, to get a pair of type B4, we should pick an entry $b$ in the $u$'s and an entry $\overline{a}$ in the $\lambda$'s such that $a>b$ and all positive integers $b+1, b+2, \dots, a-1$ belong to the $v$'s. We only have two pairs of entries that satisfy such conditions: $b=6$ and $\overline{a}=\overline{7}$, which does not have integers between them; and $b=2$ and $\overline{a}=\overline{4}$ since $3$, which is the only integer between $b$ and $a$ is in the $v$'s. It gives that $w$ covers $4\,5\,6|\overline{8}\,\overline{7}\,\overline{2}\,\overline{1}\,3$ and $2\,5\,7|\overline{8}\,\overline{6}\,\overline{4}\,\overline{1}\,3$, respectively.

Putting these cases together we have the following five covering pairs:
\begin{align*}
w=2\,5\,6|\overline{8}\,\overline{7}\,\overline{4}\,\mathbf{\overline{1}}\,3 \mbox{ and } w'_{1}=2\,5\,6|\overline{8}\,\overline{7}\,\overline{4}\, \mathbf{1} \,3 \mbox{ of type } B1; \\
w=2\,5\,6|\overline{8}\,\overline{7}\,\mathbf{\overline{4}}\,\overline{1}\,\mathbf{3} \mbox{ and } w'_{2}= 2\,5\,6|\overline{8}\,\overline{7}\,\mathbf{\overline{3}}\,\overline{1} \,\mathbf{4} \mbox{ of type } B2; \\
w=2\,\mathbf{5}\,6|\overline{8}\,\overline{7}\,\overline{4}\,\overline{1}\,\mathbf{3} \mbox{ and } w'_{3}=2\, \mathbf{3}\,6|\overline{8}\,\overline{7}\,\overline{4}\,\overline{1} \, \mathbf{5} \mbox{ of type } B3; \\
w=\mathbf{2}\,5\,6|\overline{8}\,\overline{7}\,\mathbf{\overline{4}}\,\overline{1}\,3 \mbox{ and } w'_{4}= \mathbf{4}\,5\,6|\overline{8}\,\overline{7}\,\mathbf{\overline{2}}\,\overline{1} \,3 \mbox{ of type } B4; \\
w=2\,5\,\mathbf{6}|\overline{8}\,\mathbf{\overline{7}}\,\overline{4}\,\overline{1}\,3 \mbox{ and } w'_{5}= 2\,5\,\mathbf{7}|\overline{8}\,\mathbf{\overline{6}}\,\overline{4}\,\overline{1}\,3 \mbox{ of type } B4.
\end{align*}
\end{ex} 
 
We also have a similar version of Corollary \ref{coro:dlength} where we start with any permutation $w' \in \weyl_{n}^{(k)}$ and we want to determine all possible $w \in \weyl_{n}^{(k)}$ that cover $w'$.

\begin{cor}[Length-increasing] Let $w'\in \weyl_{n}^{(k)}$. 
\begin{enumerate}
\item If $w' = \cdots\, |\, \cdots\,  1\,  \cdots$ then $w = \cdots\, |\, \cdots\,  \overline{1}\,  \cdots$ covers $w'$;
\item If $w' = \cdots\, |\, \cdots\,  \overline{a-1}\,  \cdots \, a\, \cdots$ then $w = \cdots\, |\, \cdots\, \overline{a}\,  \cdots \, (a-1)\, \cdots$ covers $w'$;
\item If $w' = \cdots\, b \,\cdots\, |\,  \cdots \, a\, \cdots$ where $a>b$ and all positive values $b+1, b+2, \dots, a-1$ belong to the $\lambda$'s then $w = \cdots\,  a \,\cdots \,|\, \cdots\, b\, \cdots$ covers $w'$;
\item If $w' = \cdots\, a \,\cdots\, |\,  \cdots \, \overline{b}\, \cdots$ where $a>b$ and all positive values $b+1, b+2, \dots, a-1$ belong to the $v$'s then $w = \cdots\,  b \,\cdots \,|\, \cdots\, \overline{a}\, \cdots$ covers $w'$.
\end{enumerate}
\end{cor}

\section{Maya diagram and dual permutations}\label{sec:maya}

The Maya diagram of a permutation $w$ in $\weyl_{n}^{(k)}$ is a row of $n$ boxes where each box is marked with a symbol $\circ$, $\bullet$, or $\times$ as following: the integers $u_{1},\dots, u_{k}$ are the positions of the boxes with $\circ$, $\lambda_{1},\dots, \lambda_{r}$ are the positions of the boxes with $\bullet$, and $v_1,\dots, v_{n-k-r}$ are the positions of the vacant boxes which will be marked with $\times$. For instance, the permutation $w= 2\, 5\, 6| \overline{8}\, \overline{7}\, \overline{4}\, \overline{1}\, 3$ in $\weyl_{8}^{(3)}$ is denoted as
$$
\begin{young}[13pt]
\bullet & \circ & \times & \bullet & \circ & \circ & \bullet & \bullet
\end{young}
$$
\vspace*{0cm}

The longest element $w_{0}$ of $\weyl^{(k)}_{n}$ is the $k$-Grassmannian permutation given by 
\begin{align*}
w_{0}=1\,2\,\cdots\, k| \overline{n\vphantom{1}}\,\overline{n-1}\, \cdots\, \overline{k+1}.
\end{align*}

Next lemma states some properties of Maya diagrams that can be easily obtained using the definition.

\begin{lem}\label{lem:mayaprop} \ \\
\vspace*{-.5cm}
\begin{enumerate}
\item The identity permutation $e$ in $ \weyl_{n}^{(k)}$ is represented as
\vspace*{-.3cm}
$$
\begin{young}[13pt]
 , \scriptstyle{1} & ,& , \scriptstyle k  & , \scriptstyle{k+1} & ,& , \scriptstyle n \\
\circ & \cdotsss & \circ & \times & \cdotsss & \times
\end{young}
$$

\item The longest element $w_{0} $ in $ \weyl_{n}^{(k)}$ is represented as
\vspace*{-.3cm}
$$
\begin{young}[13pt]
 , \scriptstyle{1} & ,& , \scriptstyle k  & , \scriptstyle{k+1} & ,& , \scriptstyle n \\
\circ & \cdotsss & \circ & \bullet & \cdotsss & \bullet
\end{young}
$$

\item For every permutation $w$ of $\weyl_{n}^{(k)}$,  the Maya diagram of $w$ contains exactly $k$ boxes marked with $\circ$;

\item Let $w\in\weyl_{n}^{(k)}$. Given $i\in [k]$, the integer $\mu_{i}$ is the number of vacant boxes $\times$ to the right of the $i$-th box marked with $\circ$.

\end{enumerate}
\end{lem}

We can use assertion (4) of Lemma \ref{lem:mayaprop} to compute the length of $w$ in its Maya Diagram. By Lemma \ref{lem:dim}, it is the sum of $\alpha$'s and $\lambda$'s. The $\lambda$'s corresponds to the sum of the positions of the $\bullet$'s. The computation of the $\alpha$'s follows by Equation \eqref{eq:alpha3} which says that $\alpha_{i}=n-k-\mu_{i}$. For instance, if $w= 2\,5\,6| \overline{8}\, \overline{7}\, \overline{4}\, \overline{1}\, 3$ then we have that 
$$
i=1: \quad
\begin{young}[13pt]
\bullet & !\circ & \times & \bullet & \circ & \circ & \bullet & \bullet
\end{young}
\quad \rightarrow\quad \mu_{1} = 1;\\
$$
$$
i=2: \quad
\begin{young}[13pt]
\bullet & \circ & \times & \bullet & !\circ & \circ & \bullet & \bullet
\end{young}
\quad \rightarrow\quad \mu_{2} = 0;\\
$$
$$
i=3: \quad
\begin{young}[13pt]
\bullet & \circ & \times & \bullet & \circ & ! \circ & \bullet & \bullet
\end{young}
\quad \rightarrow\quad \mu_{3} = 0.\\
$$

Hence, $\alpha_1=5-1=4, \alpha_2=\alpha_3=5$, $\lambda_1=1, \lambda_2=4, \lambda_3=7$, and $\lambda_4=8$ such that $\ell(w)=34$.

Given $w,w'\in \weyl_{n}^{(k)}$ such that $w$ covers $w'$, we also can denote the four types of pairs using the Maya diagram as following:
\begin{itemize}
\item {Type B1:} $w$ should contain $\bullet$ in the first position, while $w'$ contains $\times$ in the first position. This pair can be represented as
\vspace*{-.3cm}
$$
w=\begin{young}[13pt]
, \scriptstyle{1} \\
\bullet & ]= \cdotsss
\end{young}
\quad \mbox{ and }\quad
w'=\begin{young}[13pt]
, \scriptstyle{1} \\
\times & ]= \cdotsss
\end{young}
$$
\item {Type B2:} $w$ should contain $\bullet$ in position $a$ and $\times$ in position $a-1$, while $w'$ contains $\times$ in position $a$ and $\bullet$ in position $a-1$. This pair can be represented as
\vspace*{-.3cm}
$$
w=\begin{young}[13pt]
, & , \scriptstyle{a-1}& , \scriptstyle a \\
== \cdotsss & ]= \times & ]= \bullet & ]= \cdotsss
\end{young}
\quad \mbox{ and }\quad
w'=\begin{young}[13pt]
, & ,  \scriptstyle{a-1}& , \scriptstyle a\\
== \cdotsss & ]= \bullet & ]= \times & ]= \cdotsss
\end{young}
$$
\item {Type B3:} $w$ should contain $\circ$ in position $a$, $\times$ in position $a-x$, while $w'$ contains $\times$ in position $a$, $\circ$ in position $a-x$. Moreover, by Lemma \ref{lem:B3B4}, both contain $\bullet$ in all positions between $a-x$ and $a$. This pair can be represented as
\vspace*{-.3cm}
$$
w=\begin{young}[13pt]
, & ,  \scriptstyle{a-x} & ,  & , &, & , \scriptstyle a \\
=] \cdotsss & ; \times & \bullet &  \cdotsss & \bullet & ; \circ & ]= \cdotsss
\end{young}
\quad \mbox{ and }\quad
w'=\begin{young}[13pt]
, & ,  \scriptstyle{a-x} & ,  & , &, & , \scriptstyle a \\
=] \cdotsss & ; \circ & \bullet &  \cdotsss & \bullet & ; \times & ]= \cdotsss
\end{young}
$$
\item {Type B4:} $w$ should contain $\bullet$ in position $a$, $\circ$ in position $a-x$, while $w'$ contains $\circ$ in position $a$, $\bullet$ in position $a-x$. Moreover, by Lemma \ref{lem:B3B4}, both contain $\times$ in all positions between $a-x$ and $a$. This pair can be represented as
\vspace*{-.3cm}
$$
w=\begin{young}[13pt]
, & ,  \scriptstyle{a-x}& ,  & , &, & , \scriptstyle a \\
=] \cdotsss & ; \circ & \times &  \cdotsss & \times  & ; \bullet & ]= \cdotsss
\end{young}
\quad \mbox{ and }\quad
w'=\begin{young}[13pt]
, & ,  \scriptstyle{a-x}& ,  & , &, & , \scriptstyle a \\
=] \cdotsss & ; \bullet & \times &  \cdotsss & \times  & ; \circ & ]= \cdotsss
\end{young}
$$
\end{itemize}
\vspace*{0.3cm}

Using Maya diagrams to represent these pairs also give us an easier way to identify each type of pairs. In fact, if $w$ is any permutation in $\weyl_{n}^{(k)}$ then we can find all  permutations $w'$ covered by $w$ merely looking for the above patterns in the Maya diagram of $w$. In other words, we could easily rewrite Corollary \ref{coro:dlength} in terms of Maya diagrams.

For instance, consider $w= 2\, 5\, 6| \overline{8}\, \overline{7}\, \overline{4}\, \overline{1}\, 3\in\weyl_{8}^{(3)}$ of Example \ref{ex:covering}. We can obtain the same five covering pairs using only the above patterns of $w$ as it follows:
$$
B1:
\quad
w=
\begin{young}[13pt]
!\bullet & \circ & \times & \bullet & \circ & \circ & \bullet & \bullet
\end{young}
\quad \mbox{and} \quad
w'_{1}=
\begin{young}[13pt]
!\times & \circ & \times & \bullet & \circ & \circ & \bullet & \bullet
\end{young}
$$
$$
B2:
\quad
w=
\begin{young}[13pt]
\bullet & \circ & !\times & !\bullet & \circ & \circ & \bullet & \bullet
\end{young}
\quad \mbox{and} \quad
w'_{2}=
\begin{young}[13pt]
\bullet & \circ & !\bullet & !\times & \circ & \circ & \bullet & \bullet
\end{young}
$$
$$
B3:
\quad
w=
\begin{young}[13pt]
\bullet & \circ & !\times & \bullet & !\circ & \circ & \bullet & \bullet
\end{young}
\quad \mbox{and} \quad
w'_{3}=
\begin{young}[13pt]
\bullet & \circ & !\circ & \bullet & !\times & \circ & \bullet & \bullet
\end{young}
$$
$$
B4:
\quad
w=
\begin{young}[13pt]
\bullet & !\circ & \times & !\bullet & \circ & \circ & \bullet & \bullet
\end{young}
\quad \mbox{and} \quad
w'_{4}=
\begin{young}[13pt]
\bullet & !\bullet & \times & !\circ & \circ & \circ & \bullet & \bullet
\end{young}
$$
$$
B4:
\quad
w=
\begin{young}[13pt]
\bullet & \circ & \times & \bullet & \circ & !\circ & !\bullet & \bullet
\end{young}
\quad \mbox{and} \quad
w'_{5}=
\begin{young}[13pt]
\bullet & \circ & \times & \bullet & \circ & !\bullet & !\circ & \bullet
\end{young}
$$
\vspace*{0.cm}

Given $w\in \weyl^{(k)}_{n}$, define $w^{\vee}=w w_{0}$ the \emph{dual permutation of $w$}. Notice that the action of $w_{0}$ on $w$ will reverse position and sign of the last $(n-k)$ positions of $w$. In other words, if $w$ is written as in Equation \eqref{eq:kgrass}, the one-line notation of the dual permutation of $w$ is
\begin{align*}
w^{\vee} = u_1\,\cdots\, u_k | \overline{v_{n-k-r}}\,\cdots\, \overline{v_{1}}\, \lambda_{1}\,\cdots\, \lambda_{r}.
\end{align*}
Clearly, we have $w^{\vee}\in \weyl^{(k)}_{n}$. The Maya diagram of $w^{\vee}$ is given by replacing all $\bullet$'s of $w$ by $\times$'s, and replacing all $\times$'s of $w$ by $\bullet$'s. For instance, the dual of $w= 2\, 5\, 6| \overline{8}\, \overline{7}\, \overline{4}\, \overline{1}\, 3$ is the permutation $w^{\vee}=2\,5\,6|\overline{3}\, 1\,4\,7\,8$ and the Maya diagram is
$$
\begin{young}[13pt]
\times & \circ & \bullet  & \times & \circ & \circ & \times & \times
\end{young}
$$

Let us compute the length of $w^{\vee}$.
\begin{lem}
The length of $w^{\vee}$ is $\ell(w^{\vee})=\ell(w_{0})-\ell(w)$.
\end{lem}
\begin{proof}
First of all, let us compute the length of $w_{0}$. Using Equation \eqref{eq:length} we can easily show that $\ell(w_{0})=\frac{1}{2}(n+3k+1)(n-k)$ since $\mathrm{inv}(w_{0}(1),\dots, w_{0}(n))=k(n-k)$ and
\begin{align*}
-\sum_{\mathclap{\{j\mid w(j)<0\}}}\ w(j)=\frac{1}{2}(n+k+1)(n-k).
\end{align*}
Notice that the partitions associated with $w^{\vee}$ are $\alpha_{i}^{\vee}=n-k-\mu_{i}(w^{\vee})=n-k-d_{i}$ for $i\in [k]$ and $\lambda_{i}^{\vee}=v_{i}$ for $i\in[n-k-r]$. Using Equation \eqref{eq:ui}, we have
\begin{align*}
\ell(w)+\ell(w^{\vee}) & = \sum_{i=1}^{k} \alpha_{i} + \sum_{i=1}^{r} \lambda_{i} + \sum_{i=1}^{k} \alpha_{i}^{\vee} + \sum_{i=1}^{\mathclap{n-k-r}} \lambda_{i}^{\vee}\\
& = \sum_{i=1}^{k} (n-k-\mu_{i}) + \sum_{i=1}^{r} \lambda_{i} + \sum_{i=1}^{k} (n-k-d_{i}) + \sum_{i=1}^{\mathclap{n-k-r}} v_{i}\\
& = \sum_{i=1}^{k} (n-k) + \sum_{i=1}^{r} \lambda_{i} + \sum_{i=1}^{k} (n-k-d_{i}-\mu_{i}) + \sum_{i=1}^{\mathclap{n-k-r}} v_{i}\\
& = k(n-k) + \sum_{i=1}^{r} \lambda_{i} + \sum_{i=1}^{k} (u_{i}-i) + \sum_{i=1}^{\mathclap{n-k-r}} v_{i}\\
& = k(n-k) + \left(\sum_{i=1}^{r} \lambda_{i} + \sum_{i=1}^{k} u_{i} + \sum_{i=1}^{\mathclap{n-k-r}} v_{i}\right) - \sum_{i=1}^{k} i\\
& = k(n-k) + \sum_{i=1}^{n} i - \sum_{i=1}^{k} i = k(n-k) + \sum_{\mathclap{i=k+1}}^{n} i  \\
& = k(n-k) + \frac{1}{2}(n+k+1)(n-k) = \frac{1}{2}(n+3k+1)(n-k).
\end{align*}

Hence, $\ell(w)+\ell(w^{\vee})=\ell(w_{0})$.
\end{proof}

\begin{cor}
$w'\leqslant w$ if, and only if, $w^{\vee}\leqslant (w')^{\vee}$. Moreover, $\ell(w)-\ell(w')=\ell((w')^{\vee})-\ell(w^{\vee})$. In particular, $w$ covers $w'$ if, and only if, $(w')^{\vee}$ covers $w^{\vee}$.
\end{cor}

Then, the duality of a permutation also imply a duality over the covering pairs. The next proposition states a duality among the type of pairs.

\begin{prop} Let $w, w $ be permutations in $ \weyl^{(k)}_{n}$ such that $w$ covers $w'$. Then
\begin{enumerate}
\item $w,w'$ is a pair of type B1 if, and only if, $(w')^{\vee}, w^{\vee}$ is a pair of type B1;
\item $w,w'$ is a pair of type B2 if, and only if, $(w')^{\vee}, w^{\vee}$ is a pair of type B2;
\item $w,w'$ is a pair of type B3 if, and only if, $(w')^{\vee}, w^{\vee}$ is a pair of type B4;
\end{enumerate}
\end{prop}
\begin{proof}
This result can be easily obtained using the pairs $w,w'$ and $(w')^{\vee},w^{\vee}$ represented in the Maya diagrams.
\end{proof}

For instance, $w=2\,\mathbf{5}\,6|\overline{8}\,\overline{7}\,\overline{4}\,\overline{1}\,\mathbf{3}$ and $w'_{3}=2\, \mathbf{3}\,6|\overline{8}\,\overline{7}\,\overline{4}\,\overline{1} \, \mathbf{5}$ from Example \ref{ex:covering} is a pair of type B3, whereas the dual $(w'_{3})^{\vee}=2\,\mathbf{3}\,6|\mathbf{\overline{5}}\, 1\,4\,7\,8$ and $w^{\vee}=2\,\mathbf{5}\,6|\mathbf{\overline{3}}\, 1\,4\,7\,8$ is a pair of type B4.

Let $w,w'\in\weyl_{n}^{(k)}$. We write $w'\to w$ if $w$ covers $w'$. The Bruhat graph is the graph such that the set of vertices is $\weyl_{n}^{(k)}$ and the (oriented)
arrows are the covering relation for the Bruhat order.

\begin{ex}
Let $n=4$ and $k=2$. Figure \ref{fig:n4k2} represent the Bruhat graph for $\weyl_{4}^{(2)}$. The dual of a permutation in such graph is obtained by the reflection through the horizontal dashed line. The type of each covering is denoted using different arrow styles: doted arrows are pairs of type B1; dashed arrows are pairs of type B2; thinner arrows are pairs of type B3; and thicker arrows are pairs of type B4.

\begin{figure}[ht]
\centering
\includegraphics[scale=0.8]{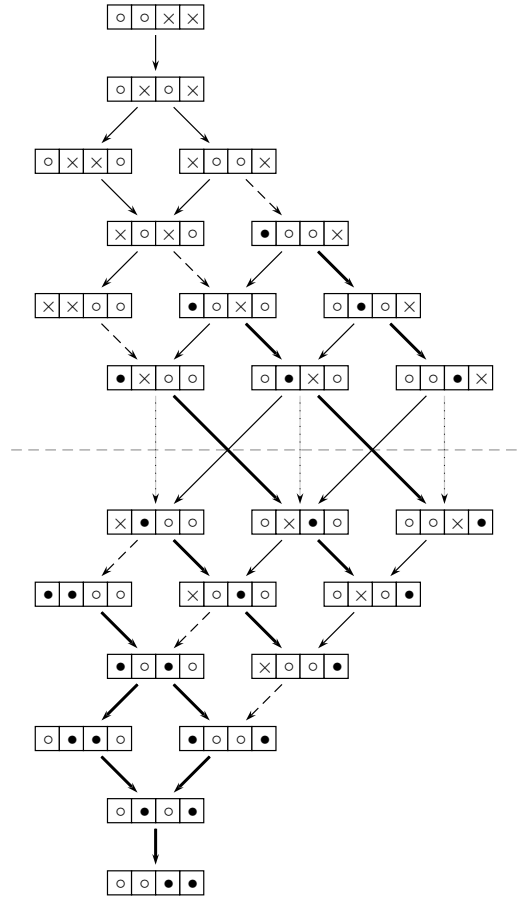}
\caption{Bruhat graph for $n=4$ and $k=2$.}
\label{fig:n4k2}
\end{figure}
\end{ex}

\end{document}